\documentclass[11pt]{article}

\title{On Symplectic Capacities of Toric Domains}
\usepackage{authblk}

\author{
Michael Landry\thanks{Yale University, \texttt{michael.landry@yale.edu}}, 
Matthew McMillan\thanks{Wheaton College, \texttt{matthew.mcmillan@my.wheaton.edu}}, 
Emmanuel Tsukerman\thanks{UC Berkeley, \texttt{e.tsukerman@math.berkeley.edu}}
}

\date{}

\usepackage{amssymb,amsfonts,amsmath,amssymb,amsthm, amscd, latexsym}

\usepackage[english]{babel} 
\usepackage{commath}

\usepackage{enumitem}
\usepackage{geometry}
\usepackage{tikz}
\usetikzlibrary{arrows}
\usetikzlibrary{patterns,snakes}
\usepackage{mathrsfs}
\usepackage{caption}
\usepackage{subcaption}

\theoremstyle{plain}
\newtheorem{Theorem}{Theorem}[section]
\newtheorem*{Theorem*}{Theorem}
\newtheorem{Lemma}[Theorem]{Lemma}
\newtheorem*{Lemma*}{Lemma}
\newtheorem{Corollary}[Theorem]{Corollary}
\newtheorem*{Corollary*}{Corollary}
\newtheorem{Proposition}[Theorem]{Proposition}
\newtheorem*{Proposition*}{Proposition}

\newtheorem{Conjecture*}{Conjecture}

\newtheorem{Claim*}{Claim}

\newtheorem*{Question*}{Question}

\theoremstyle{definition}

\newtheorem*{Definition*}{Definition}
\newtheorem{Remark}[Theorem]{Remark}
\newtheorem*{Remark*}{Remark}

\newtheorem{Assumption*}{Assumption}

\newtheorem{Disclosure*}{Disclosure}

\newtheorem{Criterion}[Theorem]{Criterion}
\newtheorem*{Criterion*}{Criterion}

\newcommand{\step}[1] {\medskip \noindent {\em Step #1.\/}}

\newcommand{\eqdef}{\;{:=}\;}

\newcommand{\C}{{\mathbb C}}

\newcommand{\R}{{\mathbb R}}

\newcommand{\op}{\operatorname}

\newcommand{\std}{\op{std}}
\newcommand{\id}{\op{id}}
\newcommand{\inte}{\op{int}}

\newcommand{\dist}{\op{dist}}
\newcommand{\ext}{\op{ext}}

\newcommand{\eps}{\varepsilon}

\newcommand{\bpm}{\begin{pmatrix}}
\newcommand{\epm}{\end{pmatrix}}

\begin{document}

	\maketitle
	
	\begin{abstract}
		A toric domain is a subset of $(\C^n,\omega_{\std})$ which is invariant under the standard rotation action of $\mathbb{T}^n$ on $\C^n$. For a toric domain $U$ from a certain large class for which this action is not free, we find a corresponding toric domain $V$ where the standard action is free, and for which $c(U)=c(V)$ for any symplectic capacity $c$. Michael Hutchings gives a combinatorial formula for calculating his embedded contact homology symplectic capacities for certain toric four-manifolds on which the $\mathbb{T}^2$-action is free. Our theorem allows one to extend this formula to a class of toric domains where the action is not free. We apply our theorem to compute ECH capacities for certain intersections of ellipsoids, and find that these capacities give sharp obstructions to symplectically embedding these ellipsoid intersections into balls.  
	\end{abstract}

	\section{Introduction}
	
	Symplectic capacities, introduced by Gromov and Hofer, are symplectic invariants that assign a nonnegative real number to a subset $U\subset (\C^n, \omega_{\std})$ and have the following properties:
	
	\begin{enumerate}
		\item[C1]
		\textbf{Monotonicity:} $ c(U)\leq c(V) $ if $ U\hookrightarrow V. $
		\item[C2]
		\textbf{Conformality:} $ c(\lambda U) = \lambda^2 c(U) $ for $ \lambda \in \R. $
		\item[C3]
		\textbf{Nontriviality:} $ 0<c(B^{2n}(1))<\infty. $
	\end{enumerate}
	Note that combining all three requires a finite capacity for any bounded $U$. Sometimes additional nontriviality and normalization axioms are also assumed, but we do not use them here. Many useful symplectic capacities have been defined; some are listed in \cite{chls}.
		
		Define the \textit{moment map} $\mu:\C^n\rightarrow \R^n$ of the symplectic manifold $(\C^n,\omega_{\std})$ by
	\[ \mu(z_1,\,\dots,z_n)=(\pi |z_1|^2,\dots, \pi |z_n|^2), \]
	where $\omega_{\std}$ is the standard symplectic form $\omega_{\std}=\sum_{i=1}^n dx_i\wedge dy_i$ on $\C^n$, and call $\mu(\C^n)$ the moment space. We call $U\subset (\C^n,\omega_{\std})$ a \textit{toric domain} when it can be written $U=\mu^{-1}(A)$ for some \textit{moment region} $A\subset \R_{\geq 0}^n$ in the moment space, or equivalently when it is invariant under the rotation action of $\mathbb{T}^n$ on $\C^n.$ Note that this is a special case of the more general moment map associated with a Hamiltonian action of a Lie group.
	
	Since these toric domains are uniquely represented by their moment regions, we will refer to a symplectic capacity $c(A)$ of a moment region $A$, and by this mean $ c(\mu^{-1}(A)) $. A simple calculation shows that C2 is equivalent to $c(\lambda A)=\lambda c(A)$.
	
		Our main theorem is that for a duly qualified toric domain $U$ whose moment region satisfies Criterion~\ref{Crit:disk-degenerate} given below, any symplectic capacity of $U$ is the same as the capacity of a toric domain with a free action, one whose moment region is $\mu(U)$ translated off the coordinate planes in the moment space.
		
	\begin{Criterion} \label{Crit:disk-degenerate}
		Let $A\subset \R_{\geq 0}^n.$ If $A$ intersects a coordinate plane $P_i=\{ (\rho_1,\,\dots,\rho_n)\in \R^n \mid \rho_i=0 \}$, then any line normal to $P_i$ has connected intersection with $A\cup P_i$.
	\end{Criterion}
	The necessary further qualifications are given in the theorem statement below. Figure~\ref{Fig:disk-degenerate} illustrates this condition for $n=2$. In this case Criterion~\ref{Crit:disk-degenerate} ensures that the toric domain is a disk bundle over its projection to the first complex plane of $\C^2$; more generally, for $A$ satisfying the other conditions below, Criterion~\ref{Crit:disk-degenerate} requires $\mu^{-1}(A)$ to be a (generalized) disk bundle over its projection to any coordinate plane $P_i$ which it touches. Disks in the fiber space degenerate to points where $A$ touches a coordinate plane.
	
		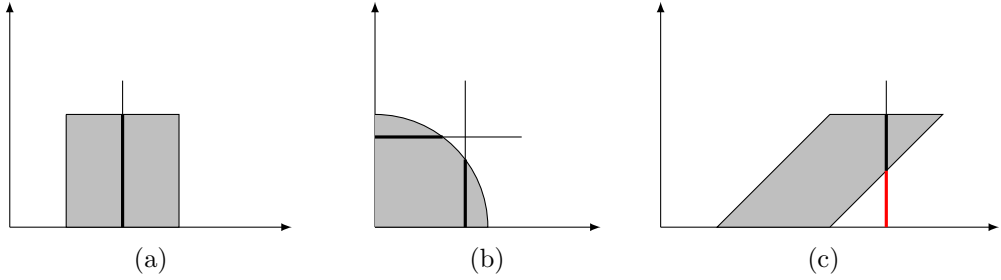
\begin{figure}[]
		\begin{subfigure}[b]{.3\textwidth}
			\centering
			\begin{tikzpicture}[>=latex,scale=1.5]
				\draw [->](0,0)--(0,2);
				\draw [->](0,0)--(2.5,0);
				\filldraw[fill=gray!50!white,even odd rule] (.5,1)--(1.5,1)--(1.5,0)--(.5,0)--(.5,1);
				\draw[very thick] (1,0)--(1,1);
				\draw (1,1)--(1,1.3);
			\end{tikzpicture}
			\caption{}
			\label{Fig:d-da}
		\end{subfigure}
		\begin{subfigure}[b]{.3\textwidth}
			\centering
			\begin{tikzpicture}[>=latex,scale=1.5]
				\draw [->](0,0)--(0,2);
				\draw [->](0,0)--(2,0);
				\filldraw[fill=gray!50!white] (0,0)--(1,0) arc (0:90:1cm);
				\draw[very thick] (.8,0)--(.8,.6);
				\draw (.8, .6)--(.8,1.3);
				\draw[very thick] (0,.8)--(.6,.8);
				\draw (.6,.8)--(1.3,.8);
			\end{tikzpicture}
			\caption{}
			\label{Fig:d-db}
		\end{subfigure}
		\begin{subfigure}[b]{.3\textwidth}
			\centering
			\begin{tikzpicture}[>=latex,scale=1.5]
				\draw [->](0,0)--(0,2);
				\draw [->](0,0)--(3,0);
				\filldraw[fill=gray!50!white,even odd rule] (1.5,1)--(2.5,1)--(1.5,0)--(.5,0)--(1.5,1);
				\draw[very thick] (2,.5)--(2,1);
				\draw (2,1)--(2,1.3);
				\draw[red, very thick] (2,0)--(2,.5);
			\end{tikzpicture}
			\caption{}
			\label{Fig:d-dc}
		\end{subfigure}
		\caption{Appropriate moment regions. (a), (b) satisfy the conditions of Criterion~\ref{Crit:disk-degenerate}, (c) does not.}
		\label{Fig:disk-degenerate}
	\end{figure}

	\begin{Theorem} \label{Thm:disk-degenerate}
		Let $A\subset \R^n_{\geq 0}$ be a moment region which is compact with star-shaped interior, and whose boundary intersects transversely the rays from the star-center. Assume that $A$ satisfies Criterion~\ref{Crit:disk-degenerate}. Then $c(A) = c(A+(1,1,\dots,1))$ for any symplectic capacity $c.$
	\end{Theorem}
	The theorem is proved by establishing equal lower and upper bounds on $c(A)$ in terms of $c(A+(1,1,\dots,1)).$ The lower bound follows readily from properties of toric domains and the axioms C1-C3, but for the upper bound we must combine the axioms with a nontrivial symplectic embedding. Since the proof assumes only the general axioms for capacities, this result holds for all symplectic capacities. Note that the action on a given toric domain $U=\mu^{-1}(A)$ is free if and only if $U$ does not intersect the origin in any $\mathbb{C}$ factor; that is, its moment region does not touch any coordinate plane $P_i=\{ (\rho_1,\,\dots,\rho_n)\in \R^n \mid \rho_i=0 \}$ in the moment space.

	The embedded contact homology (ECH) developed by Michael Hutchings provides a natural way to define certain symplectic capacities called ECH capacities. They are defined for any subset of a symplectic 4-manifold. In~\cite{hutchings}, Hutchings gives a combinatorial method to compute these capacities for toric domains over convex moment regions that do not touch the axes of the moment space $\R^2_{\geq 0}$ (that is, the torus action is free). This method is presented in~Section~\ref{AppECH}. In \cite[Rmk. 4.15]{hut1} and \cite[\S 1.2]{gardinerconcave}, it was conjectured that Hutchings's formula should remain true in most, and probably all, cases where $\mu(U)$ does touch one or both axes. Theorem~\ref{Thm:disk-degenerate} shows that this is true for the ECH capacities of a large class of toric domains by showing that it is true for all symplectic capacities.
	
		Given $a,b \in \R^+$, define the ellipsoid
	\begin{equation}
		\label{eqn:ellipsoid}
		E(a,b) \eqdef \left\{(z_1,z_2)\in\C^2 \;\bigg|\; \frac{\pi|z_1|^2}{a} +
		  \frac{\pi|z_2|^2}{b}\le 1\right\},
		\end{equation}
	the ball
	\[
	B(a) \eqdef E(a,a),
	\]
	and the polydisk
	\begin{equation}
		\label{eqn:polydisk}
		P(a,b)\eqdef \left\{(z_1,z_2)\in\C^2 \;\big|\; \pi|z_1|^2 \le a, \;\;
		  \pi|z_2|^2 \le b\right\},
		\end{equation}
	where each inherits the standard symplectic form from $\C^2$.
	
	In Section~\ref{AppECH} we use Theorem~\ref{Thm:disk-degenerate} to compute ECH capacities of a class of intersections of ellipsoids. We also study symplectic embeddings of domains from this class, proving the following proposition:
	
	\begin{Proposition} \label{Prop:non-squeezing}
		Let $R$ be the radius of the smallest ball containing $E(a,b)\cap E(c,d)$, $a<b,c>d$, and let $\rho=\inf\{r\mid E(a,b)\cap E(c,d)\hookrightarrow B(r)\}$. If $2a, 2d\ge R$, then $\rho=R$.
	\end{Proposition}
	
	It is known that ECH capacities provide sharp obstructions to symplectically embedding ellipsoids into ellipsoids (proved by McDuff~\cite{mcduff2}), and ellipsoids into polydisks (Frenkel-M\"uller~\cite{frenkel}). Recall that by Gromov's non-squeezing theorem \cite{gromov}, a ball symplectically embeds into a cylinder in $\R^{2n}$ if and only if the radius of the cylinder exceeds that of the ball. This is an illustration of symplectic rigidity, and is easily recovered from the ECH capacities on these domains. The computation of ECH capacities of the ellipsoid intersections above shows that they give sharp obstructions to symplectically embedding those ellipsoid intersections into balls. Since the balls have much larger volume than the ellipsoid intersections, Proposition~\ref{Prop:non-squeezing} is another example of symplectic rigidity.

	In Proposition~\ref{Prop:non-squeezing}, the ECH capacities give a sharp obstruction. Recent work of Hind and Lisi \cite{hind} shows that neither ECH capacities nor Ekeland-Hofer capacities give sharp obstructions to symplectic embeddings of arbitrary toric domains; in particular the ECH and Ekeland-Hofer obstructions to symplectically embedding a product of polydisks into a ball are not always sharp. The torus action on polydisks and balls is not free, so we might ask whether the situation is any different if we consider only toric domains for which the action is free. However, the case of free torus action is not different in this way, as the following corollary of Theorem~\ref{Thm:disk-degenerate} shows:
	
	\begin{Corollary} \label{Cor:NonSharp}
	Let $P^*(1,2)=\mu^{-1}(\mu(P(1,2))+(1,1))$ be a toric domain, let $a<3,$ and let $B^*(a)=\mu^{-1}(\mu(B^4(a))+(1,1)).$ There is no symplectic embedding $P^*(1,2)\hookrightarrow B^*(a).$
	\end{Corollary}
	
	This shows that neither ECH nor Ekeland-Hofer capacities are sharp even when we consider only toric domains with a free action because the obstruction given by both of these sequences of capacities is $a\geq 2$ (see \cite{hind}). This corollary is proved in Section~\ref{FreeDomains}.

	\section{Proof of main theorem} \label{MainTheorem}
		
		In this section we prove Theorem~\ref{Thm:disk-degenerate} by constructing symplectomorphisms as the products of area preserving maps. It will be convenient to have the following standard lemma, which shows that translations in the moment space induce symplectomorphisms on toric domains whose moment regions do not touch any coordinate plane.
		\begin{Lemma}
		Suppose $U\subset (\R^{2n},\omega_{\std})$ is a toric domain with free torus action such that $\mu(U)=A$, and $B$ is any translate of $A$ such that the torus action on $\mu^{-1}$ is also free. Then $U$ and $V=\mu^{-1}(B)$ are symplectomorphic. In particular, they have the same symplectic capacity for any capacity.
		\label{Lem:translates}
	\end{Lemma}
	\begin{proof}
		We can parametrize $U$ by $ g:A\times \mathbb{T}^n \rightarrow U$ defined by
		\[ g(\rho_1,\dots,\rho_n,e^{i\theta_1}, \dots, e^{i\theta_n}) = (\sqrt{\frac{\rho_1}{\pi}}e^{i\theta_1},\dots, \sqrt{\frac{\rho_n}{\pi}}e^{i\theta_n}). \] Then we can pull back the standard symplectic form to $A\times \mathbb{T}^n$. A simple calculation shows that for the first term, \[g^*(dx_1\wedge dy_1) = \frac{1}{2\pi} d\rho_1\wedge d\theta_1,\]
		thus
		\[g^*\omega_{std} = \frac{1}{2\pi} \sum_{i=1}^{n} d\rho_i\wedge d\theta_i.\]

		It is clear that translation in moment space does not affect this last form, so conjugating a translation by this parametrization yields the desired symplectomorphism.
	\end{proof}
	
	Another important fact that can be seen from the proof of Lemma~\ref{Lem:translates} is that for a toric domain $U$ with free torus action and moment region $A$, the symplectic volume of $U$ is equal to the volume of $A$:
	\begin{align*}
		\op{vol}(U,\omega_{\std}) &=\frac{1}{n!}\int_U\omega_{\std}^n = \frac{1}{n!}\int_{A\times \mathbb{T}^n} (g^*\omega_{\std})^n \\
		&= \frac{1}{(2\pi)^n}\int_{A\times \mathbb{T}^n} d\rho_1\wedge \dots \wedge d\rho_n\wedge d\theta_1\wedge \dots \wedge d\theta_n \\
		&= \int_A d\rho_1\wedge \dots \wedge d\rho_n = \op{vol}(A).
	\end{align*}
	So a symplectic embedding of toric domains $ U\hookrightarrow V $ may be possible only if $ \op{vol}(\mu(U)) \leq \op{vol}(\mu(V)). $

	We will also use the following version of the ``Traynor trick'':
	\begin{Lemma}[cf. Proposition 5.2 of Traynor~\cite{traynor}] \label{Lem:wrap}
		Given $\eps>0$, there exists an area preserving diffeomorphism $\Psi :B^2(1)\rightarrow SD^2(1+\eps)=B^2(1+\epsilon)-\{x+i y \mid y=0, x\geq 0\}$ from the disk to the slit-disk such that \[ \delta<|\Psi(z)|^2<|z|^2+\eps \] for some $\delta>0$.
	\end{Lemma}
	
	\begin{figure}[]
		\centering
		\begin{tikzpicture}[scale=.8]
			\draw [densely dashed](0,0) circle (4.3);
			\draw[densely dashed] (0,0)--(4.3,0);
			\draw (-.5,0) circle (.15); 
			\draw (-.9,0) arc (180:110:.9); 
			\draw (-.9,0) arc (180:250:.9);
			\draw (-.3078, .8457) .. controls (0.1,.97) and (.4,.72) .. (.1,.5); 
			\draw (-.3078, -.8457) .. controls (0.1,-.97) and (.4,-.72) .. (.1,-.5);
			\draw (.1,.5) .. controls ( -.17,.3) and ( -.2,.1)  .. (-.2,0);
			\draw (.1,-.5) .. controls ( -.17,-.3) and ( -.2,-.1)  .. (-.2,0);
			\draw (-1.2,0) arc (180:75:1.2); 
			\draw (-1.2,0) arc (180:285:1.2);
			\draw (.3106, 1.159) .. controls (.8, 1) and ( .8, .4 ) .. (.4,.25);
			\draw (.4,.25) .. controls (.15,.18) and (-.14,.14) .. (-.14,0);
			\draw (.3106, -1.159) .. controls (.8, -1) and ( .8, -.4 ) .. (.4,-.25);
			\draw (.4,-.25) .. controls (.15,-.18) and (-.14,-.14) .. (-.14,0);
			\draw (-1.5,0) arc (180:25:1.5); 
			\draw (-1.5,0) arc (180:335:1.5);
			\draw (1.359,.6339) .. controls ( 1.55,-.1 ) and ( -.12,.2 ) .. (-.12,0);
			\draw (1.359,-.6339) .. controls ( 1.55,.1 ) and ( -.12,-.2 ) .. (-.12,0);
			\draw (-1.8,0) arc (180:15:1.8); 
			\draw (-1.8,0) arc (180:345:1.8);
			\draw (1.739,.4659) .. controls ( 1.8,.19 ) and ( 1.7,.2 ) .. (1.55,.15);
			\draw (1.55,.15) .. controls ( 1,.03 ) and ( -.1,.12 ) .. (-.1,0);
			\draw (1.739,-.4659) .. controls ( 1.8,-.19 ) and ( 1.7,-.2 ) .. (1.55,-.15);
			\draw (1.55,-.15) .. controls ( 1,-.03 ) and ( -.1,-.12 ) .. (-.1,0);
			\draw (-2.1,0) arc (180:10:2.1); 
			\draw (-2.1,0) arc (180:350:2.1);
			\draw (2.068, .3647) .. controls ( 2.1, .1 ) and ( 2 ,.12  ) .. (1.8,.1);
			\draw (1.8,.1) .. controls ( 1.3, .057 ) and ( -.08 ,.075 ) .. (-.08,0);
			\draw (2.068, -.3647) .. controls ( 2.1, -.1 ) and ( 2 ,-.12  ) .. (1.8,-.1);
			\draw (1.8,-.1) .. controls ( 1.3, -.057 ) and ( -.08 ,-.075 ) .. (-.08,0);
			\draw (-2.4,0) arc (180:5:2.4);
			\draw (-2.4,0) arc (180:355:2.4);
			\draw (5:2.4) .. controls ( 1: 2.4 ) and (2 : 2.4) .. (2:2.1);
			\draw (2:2.1) .. controls (2:1.8) and (-.07,.07) .. (-.07,0);
			\draw (-5:2.4) .. controls ( -1: 2.4 ) and (-2 : 2.4) .. (-2:2.1);
			\draw (-2:2.1) .. controls (-2:1.8) and (-.07,-.07) .. (-.07,0);
			\draw (-2.7,0) arc (180:2:2.7);
			\draw (-2.7,0) arc (180:358:2.7);
			\draw (2:2.7) .. controls (1:2.7) and (1:2.7) .. (1:2.4);
			\draw (1:2.4)--(0,.02);
			\draw (-2:2.7) .. controls (-1:2.7) and (-1:2.7) .. (-1:2.4);
			\draw (-1:2.4)--(0,-.02);
			\draw (-3,0) arc (180:2:3);
			\draw (-3,0) arc (180:358:3);
			\draw (2:3) .. controls (.5:3) and (1:3) .. (1:2.5);
			\draw (-2:3) .. controls (-.5:3) and (-1:3) .. (-1:2.5);
			\draw (-3.3,0) arc (180:2:3.3);
			\draw (-3.3,0) arc (180:358:3.3);
			\draw (2:3.3) .. controls (0:3.3) and (1:3.3) .. (.9:2.5);
			\draw (-2:3.3) .. controls (0:3.3) and (-1:3.3) .. (-.9:2.5);
			\foreach \x in {0,.3,.6}
			\draw (1.8:3.6+\x) .. controls (0.3:3.6+\x) and (.6:3.65+\x) .. (.85:2.7);
			\foreach \x in {0,.3,.6}
			\draw (-1.8:3.6+\x) .. controls (-0.3:3.6+\x) and (-.6:3.65+\x) .. (-.85:2.7);
			\foreach \x in {0,.3,.6}
			\foreach \y in {1.8,358.2}
			\draw (-3.6-\x,0) arc (180:\y:3.6+\x);
		\end{tikzpicture}
		\caption{Family of loops defining a symplectomorphism $B^2(1)\rightarrow SD(1+\eps)$.}
		\label{Fig:slit-disk}
	\end{figure}
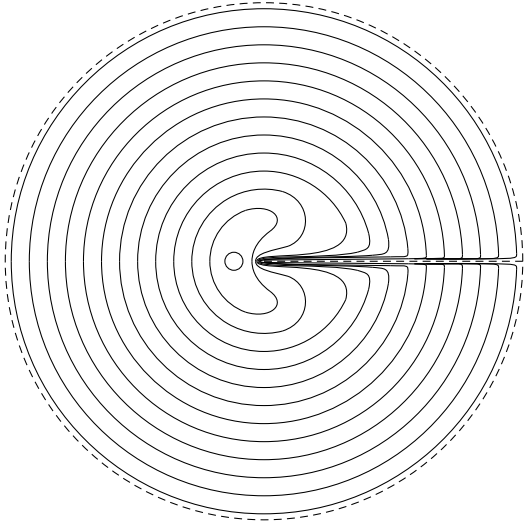	
	
	\begin{proof}
		The left inequality follows from continuity (given such a map). For existence and the right inequality, define a family of loops which avoid the slit as in Figure~\ref{Fig:slit-disk}, and apply Schlenk~\cite[Lemma~3.1]{schlenk}.
	\end{proof}
	
	With these tools we can prove the main Theorem~\ref{Thm:disk-degenerate}.
	
	\begin{proof}[Proof of Theorem~\ref{Thm:disk-degenerate}]
	
		Our technique is to find upper and lower bounds on $c(A)$ by producing symplectic embeddings and applying C1 and C2. We show that these bounds agree with each other and with $c(A+(1,1,\dots,1))$.
		
		For what follows, we define the scaling of $\R^n$ by $\lambda>0$ from $p\in \R^n$ to be the map $q\mapsto \lambda(q-p)+p$. Since $\lambda (q-p) + p = \lambda q + (1-\lambda)p$, any scaling by $\lambda$ from $p$ is equivalent to a scaling from the origin by $\lambda$ followed by translation by $(1-\lambda)p$. So with Lemma~\ref{Lem:translates} we may apply conformality of capacities, axiom C2, on moment regions scaled from points other than the origin. The reason for the requirement that rays from the star-center be transverse to the boundary will become clear in Step~2 with the scaling argument.
		
		\step{1} The lower bound may be computed as follows. Let $p$ be a star-center of $\inte A$, which means that any other point in $\inte A$ may be connected to $p$ by a line contained in $\inte A.$ Given any $\lambda<1$, let $A_{\lambda}$ be the image of $A$ under the scaling of the moment space towards $p$ by $\lambda$. Since $p$ is away from the coordinate planes, $A_{\lambda}$ is bounded away from the coordinate planes and contained in $A$. By Lemma~\ref{Lem:translates} and conformality, $ c(A_{\lambda}) = \lambda c(A+(1,1,\dots,1)) $. Then by monotonicity, $ \lambda c(A+(1,1,\dots,1)) \leq c(A) $, and since $ \lambda<1 $ was arbitrary,
		\[ c(A+(1,1,\dots,1))\leq c(A). \]
		
		\step{2} For the upper bound, we embed $A$ into an expanded version of $A,$ and apply an area-preserving map in each dimension in which $A$ touches a coordinate plane $P_i$. We will assume that $A$ is compact, star-shaped, and that the rays from a star-center $p$ intersect each $\partial A_j$ transversely.
		
		Assume without loss of generality that $A$ touches the first $k$ coordinate planes, and does not touch the others.  Let $p=(\rho_1,\,\dots,\rho_n)$ be the star-center in $A$ noted above. The projection $ \tilde{p}_1 = (0,\rho_2,\,\dots,\rho_n) $ is also a star-center: Choose any other point $q=(x_1,\,\dots,x_n)\in A$. The line from $\tilde{p}_1$ to $q$ is entirely below that from $p$ to $q$ in the $\rho_1$ coordinate. By Criterion~\ref{Crit:disk-degenerate}, any perpendicular dropped from a point in $A$ to $P_1$ remains in $A$. Hence the line from $\tilde{p}_1$ to $q$ is also in $A$, so $\tilde{p}_1$ is a star-center. Repeating in the first $k$ coordinates, we find that $\tilde{p}_k = (0,\dots,\rho_{k+1},\,\dots,\rho_n)$ is a star-center; call this point $\tilde{p}.$ A simple geometric argument making use of Criterion~\ref{Crit:disk-degenerate} shows that the rays from $\tilde{p}$ must also be transverse to each $\partial A_j;$ we omit that here.
		
		The next step will be to expand $A$ to $A_\lambda$ by a finite factor of $\lambda.$ In order to prevent $A_\lambda$ from colliding with coordinate planes, first translate $A$ away from the coordinate planes $P_{k+1}$ through $P_n$ by some large amount. Note that this is possible because by assumption $p_i>0$ for $i>k,$ and furthermore translation in the moment spaces induces a symplectomorphism. So we shall instead compute the capacity of this translate, and relabel it $A.$ Now let $A_\lambda$ be the scaling of $A$ from $\tilde{p}$ by a small $\lambda>1.$
		
		We show that $A\subset \inte A_\lambda.$ Consider any point $q=(x_1,\dots,x_n)\in A$. If $q\in \inte A$ then $q\in \inte A_\lambda,$ so suppose $q\in \partial A.$ Write $q_{1/\lambda}$ for the point mapped to $q$ under the scaling; $q_{1/\lambda}$ will be between $\tilde{p}$ and $q.$ Now since the ray from $\tilde{p}$ to $q$ is transverse to $\partial A,$ it follows that $q_{1/\lambda}$ must be in $\inte A,$ so we can find an open ball $U$ around $q_{1/\lambda}.$ That ball maps under the scaling to $U_\lambda,$ which is an open ball around $q$ in $A_\lambda.$ Thus $q\in \inte A_\lambda,$ and $A\subset \inte A_\lambda.$
		
		Let $\ext A_{\lambda}$ denote the exterior of $A_\lambda$ in $\R^n_{\geq 0}.$ Both $A$ and $A_\lambda$ are compact, so there is some $d$ so that $ 0< d < d_{\lambda} = \frac{1}{2}\dist(A, \ext A_{\lambda}) $. Now $A$ is bounded, so let $a$ be the maximum of the $\rho_1$ coordinate of $A$, and choose $\eps>0$ so that $\eps<d$. Then by Lemma~\ref{Lem:wrap}, there exists $\Psi_{a}:B^2(a)\rightarrow SD^2(a+\eps)$ such that
		\begin{equation} \label{Eqn:bounds}
		\delta<|\Psi_a(z)|^2<|z|^2+\eps
		\end{equation}
		for $\delta>0$. Let $ F_{\eps} = \Psi_{a}\times \id \times \dots \times \id $.
		
		Set $ B = \mu\circ F_{\eps} ( \mu^{-1}(A) )$. Then we claim $ B\subset \inte A_\lambda$. Consider a point $(z_1, \dots, z_n)\in \mu^{-1}(A),$ and let
		\[ (\rho_1, \dots, \rho_n)\equiv \mu(z_1, \dots, z_n) \in A. \]
		By the inequality above, $ \mu\circ F_{\eps}((z_1,\,\dots,z_n)) = (\tilde{\rho_1},\,\dots,\tilde{\rho_n}) $ where $ \tilde{\rho_1}<\rho_1 + \eps $ and $ \tilde{\rho_i}=\rho_i$ for $i>1$. Thus every point in $\mu^{-1}(A)$ is carried by $F_{\eps}$ to a point less than $d$ away from $A$, so $ B\subset \inte A_\lambda $; moreover $ \dist(B,\, \ext A_\lambda)>d_{\lambda} $. Then let $\delta = \frac{1}{2}\min \{\delta,\, d_{\lambda}\} $ and $\gamma = \lambda \delta$ (using $\lambda<2$). Set $A_{\lambda}' = A_\lambda + (\gamma,0,\dots,0).$ The lower bound on the left of Equation~\ref{Eqn:bounds}, together with the distance from $B$ to outside $A_\lambda$, show that in fact $ B\subset A_{\lambda}' $. So by Lemma~\ref{Lem:translates}, $ c(B)\leq c(A_{\lambda}') = \lambda c(A+(\delta,0,\dots,0)) $. Now $\lambda>1$ was arbitrary, so $ c(B)\leq c(A+(\delta,0,\dots,0)) $. Since $A$ and $B$ are symplectomorphic,
		\[ c(A)\leq c(A+(\delta,0,\dots,0)). \]
		
		Repeating the same process in the dimensions up to $k$ and translating up by $\delta$ in the other coordinates shows that for some $\delta>0$, $c(A)\leq c(A+(\delta, \delta, \dots, \delta))$. Combining with the lower bound, and using Lemma~\ref{Lem:translates},
		\[ c(A) = c(A+(1,1,\dots,1)). \qedhere \]
		
	\end{proof}
		
	\begin{Remark}
	It is worth noting that we may like to consider regions $A$ for which $\partial A$ is not completely smooth. The ellipsoid intersections below are one example. The notion of transversality must then be generalized slightly with the goal that $A\subset \inte A_\lambda.$ If $\partial A$ is the gluing of multiple hypersurfaces, it is sufficient that the rays from the star-center be transverse to each of the hypersurfaces at the points where they are glued together.
	\end{Remark}
	
	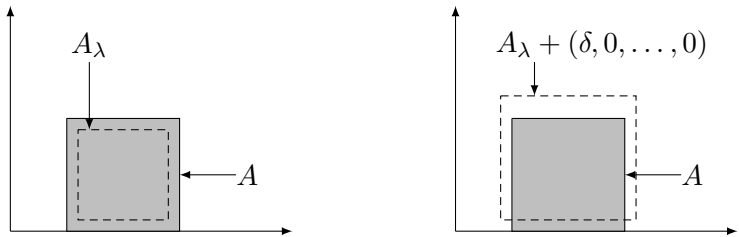
\begin{figure}
		\centering
		\begin{subfigure}[b]{.4\textwidth}
			\centering
			\begin{tikzpicture}[>=latex,scale=1.5]
				\draw [->](0,0)--(0,2);
				\draw [->](0,0)--(2.5,0);
				\filldraw[fill=gray!50!white,even odd rule] (.5,1)--(1.5,1)--(1.5,0)--(.5,0)--(.5,1);
				\draw[densely dashed] (.6,.9)--(1.4,.9)--(1.4,.1)--(.6,.1)--(.6,.9);
				\draw[<-] (.7,.9)--(.7,1.5);
				\node at (.7,1.65){$A_\lambda$};
				\draw[<-] (1.5,.5)--(2,.5);
				\node at (2.1,.5){$A$};
			\end{tikzpicture}
		\end{subfigure}
		\begin{subfigure}[b]{.4\textwidth}
			\centering
			\begin{tikzpicture}[>=latex,scale=1.5]
				\draw [->](0,0)--(0,2);
				\draw [->](0,0)--(2.5,0);
				\filldraw[fill=gray!50!white,even odd rule] (.5,1)--(1.5,1)--(1.5,0)--(.5,0)--(.5,1);
				\draw[densely dashed] (.4,1.2)--(1.6,1.2)--(1.6,0.1)--(.4,0.1)--(.4,1.2);
				\draw[<-] (.7,1.2)--(.7,1.5);
				\node at (1.3,1.65){$A_\lambda + (\delta,0,\dots,0)$};
				\draw[<-] (1.5,.5)--(2,.5);
				\node at (2.1,.5){$A$};
			\end{tikzpicture}
		\end{subfigure}
		\caption{Illustration of the conformality argument for the lower bound (left) and the upper bound (right).}
	\end{figure}
	
\section{Applications} \label{AppECH}

	\subsection{ECH capacities}
	
	The remainder of this paper focuses on 4-dimensional toric domains, with accompanying planar moment regions. Using Michael Hutchings's theory of embedded contact homology (ECH), one can associate real numbers \[ 0=c_0(M)\le c_1(M)\le c_2(M)\le \cdots \] called \emph{ECH capacities} to any 4-dimensional ``Liouville domain'' $M,$ such that each $c_i$ is a symplectic capacity for 4-manifolds. For precise definitions of ECH capacities and Liouville domains, see \cite{hutchings}.
	
	We briefly describe the computation of ECH capacities, as given by Theorem 4.14 of \cite{hut1}. Given a convex body $A$ in the moment space which does not touch any coordinate plane, we can define a norm $\ell_A$, not necessarily symmetric, as follows. Choose an origin in $A$ from which to draw position vectors to $\partial A.$ Let $v_i$ be some vector, and $q_i$ one of the position vectors on $\partial A$ such that the outward normal to $\partial A$ at $q_i$ is parallel to $v_i.$ If $v_i$ has angle between the normals to $\partial A$ at two incident edges of $\partial A$, let $q_i$ be the corner where the edges meet. Then set $\ell_A(v_i)=v_i\cdot q_i.$ It is not hard to check that this yields a well-defined norm; see \cite{hut1} for details.
	
	We compute the ECH capacities according to \cite{hutchings} as follows: for each $k$, $c_k(A)$ is the shortest perimeter length of an oriented lattice-polygon enclosing $k+1$ lattice points, where perimeter length is measured in the norm $\ell_A$ on the edge vectors of the oriented polygon.

	\subsubsection{Embedding ellipsoid intersections into balls}
	
	We now use Theorem~\ref{Thm:disk-degenerate} to compute the second ECH capacity of a family of ellipsoid intersections. This capacity is in turn used to prove Proposition~\ref{Prop:non-squeezing}. Throughout this section, let $a,b,c,d>0$, $a<b$, $c>d$, and put $R=\frac{abc+bcd-acd-abd}{bc-ad}$ (see Figure~\ref{Fig:non-squeezing}). We show that for $2a, 2d\ge R$, $c_2(E(a,b)\cap E(c,d))=R.$ A simple consequence is that $E(a,b)\cap E(c,d)$ symplectically embeds into a ball if and only if it embeds by inclusion (that is, Proposition~\ref{Prop:non-squeezing}). While in principle that result only requires the easier lower bound of Theorem~\ref{Thm:disk-degenerate}, we illustrate the use of Theorem~\ref{Thm:disk-degenerate} to produce the actual ECH capacity, which is sufficient to prove the proposition.
	
	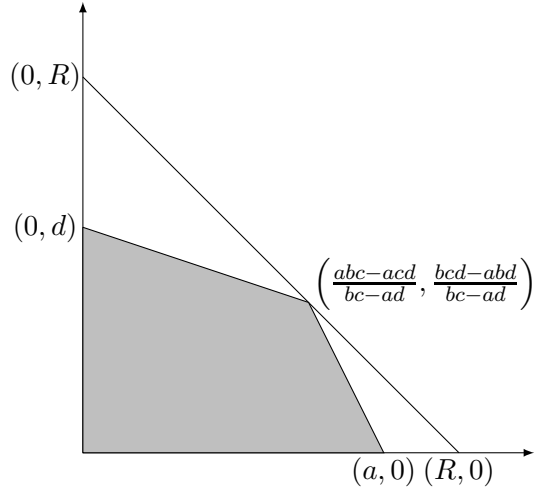
\begin{figure}[here]
		\centering
		\begin{tikzpicture}[>=latex]
		\draw [->](0,0)--(0,6);
		\draw [->](0,0)--(6,0);
		\draw (0,5)--(5,0);
		\filldraw[fill=gray!50!white,even odd rule] (0,3)--(3,2)--(4,0)--(0,0)--(0,3);
		\node at (-0.5,3){$(0,d)$};
		\node at (4,-0.25){$(a,0)$};
		\node at (-.5, 5) {$(0,R)$};
		\node at (5, -.25){$(R,0)$};
		\node at (4.55, 2.25) {$\left(\frac{abc-acd}{bc-ad}, \frac{bcd-abd}{bc-ad}\right)$};
		\end{tikzpicture}
		\caption{The image of $E(a,b)\cap E(c,d)$ under $\mu$ with suitable $a,b,c,d$, and the smallest ball into which it symplectically embeds.}
		\label{Fig:non-squeezing}
	\end{figure}
	
	A short computation, or consideration of Figure~\ref{Fig:non-squeezing}, shows that $B(R)$ is indeed the smallest ball into which $E(a,b)\cap E(c,d)$ embeds by inclusion. We first prove the following lemma:
	
	\begin{Lemma} \label{Lem:echlemma}
		If $2a, 2d\ge R$, then $c_2( E(a,b)\cap E(c,d) )=R$.
	\end{Lemma}
	
	Assuming Lemma~\ref{Lem:echlemma}, observe that Proposition~\ref{Prop:non-squeezing} is immediate:
	
	\begin{proof}[Proof of Proposition~\ref{Prop:non-squeezing}]
		By Corollary 1.3 of \cite{hutchings}, $c_2(B(r))=r$, so we have $\rho\ge R$ by Lemma~\ref{Lem:echlemma}. Since $E(a,b)\cap E(c,d)\subset B(R),$ $\rho\leq R$ and the result follows.
	\end{proof}
	
	\begin{proof}[Proof of Lemma~\ref{Lem:echlemma}]
		Let $A$ be the moment region of $E(a,b)\cap E(c,d)$. Since $A$ satisfies Criterion~\ref{Crit:disk-degenerate}, we know that $c_2(A)=c_2(A') $ for $A' = A + (1,1)$. 
				
		First we observe that the oriented lattice-polygonal path in Figure~\ref{Fig:minimalpathpic} has $\ell_{A'}$-length $R$ when oriented clockwise, so $c_2(A)\le R$.

		Let $\Gamma$ be an oriented lattice path containing 3 lattice points with edge vectors $(\alpha, \beta), (\gamma, \delta), (\epsilon, \zeta)$ (if $\Gamma$ has only two edge vectors, i.e., is just a line segment, the forthcoming argument applies \emph{mutatis mutandis}). Suppose for a contradiction that $\ell_{A'}(\Gamma)<R$. 
		
		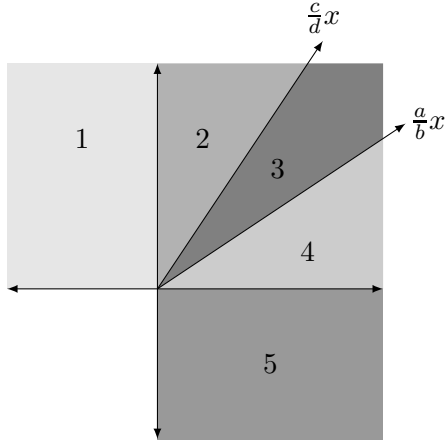
\begin{figure}[]
			\centering
			\begin{tikzpicture}[>=latex, scale=1]
			\fill[fill=black!10!white] (0,0)--(0,3)--(-2,3)--(-2,0)--(0,0);
			\fill[fill=black!30!white] (0,0)--(2,3)--(0,3)--(0,0);
			\fill[fill=black!50!white] (0,0)--(3,2)--(3,3)--(2,3)--(0,0);
			\fill[fill=black!20!white] (0,0)--(3,2)--(3,0)--(0,0);
			\fill[fill=black!40!white] (0,0)--(3,0)--(3,-2)--(0,-2)--(0,0);
			\draw[<->] (-2,0)--(3,0);
			\draw[<->] (0,-2)--(0,3);
			
			\node at (-1,2){ {1}};
			\node at (.6,2){ {2}};
			\node at (1.6,1.6){ {3}};
			\node at (2,.5){ {4}};
			\node at (1.5,-1){ {5}};		
			\draw[->] (0,0)--(2.2,3.3);
			\draw[->] (0,0)--(3.3,2.2);
			
			\node at (2.2,3.6){$\frac{c}{d}x$};
			\node at (3.6,2.2){$\frac{a}{b}x$};
			
			\end{tikzpicture}
			\caption{Calculation of $\ell_{A'}$-length by region.}
			\label{Fig:norm-regions}
		\end{figure}
		
		We first claim that $\beta,\,\delta,\,\zeta \le 1$ and that at most one is positive. Suppose without loss of generality that $\beta\ge 2$. Depending on the region in which $ (\alpha,\beta) $ lies (or its slope $ \frac{\beta}{\alpha} $, Figure~\ref{Fig:norm-regions}), the $ \ell_{A'} $-length is determined by cases:
		
		\[
		\ell_{A'}((\alpha,\beta)) = \begin{cases}
		(\alpha,\beta)\cdot (0,d)& \text{Region 1,\,2: $\alpha\le 0$ or $\frac{\beta}{\alpha}\ge\frac{c}{d}$} \\
		(\alpha,\beta)\cdot \big(\frac{abc-acd}{bc-ad}, \frac{bcd-abd}{bc-ad}\big)&\text{Region 3:  $\frac{c}{d}\le \frac{\beta}{\alpha}\le \frac{a}{b}$}\\
		(\alpha,\beta)\cdot(a,0)&\text{Region 4: $0<\frac{\beta}{\alpha}\le\frac{a}{b}$.}
		\end{cases}
		\]
		
		We treat each case separately. In Region 1, we have $(\alpha,\beta)\cdot(0,d)=\beta d\ge 2d\ge R$, a contradiction. In Region 2, $\ell_{A'}((\alpha,\beta))=(\alpha,\beta)\cdot \big(\frac{abc-acd}{bc-ad}, \frac{bcd-abc}{bc-ad}\big)$ and $\alpha\ge 1$. Hence $(\alpha,\beta)\cdot \big(\frac{abc-acd}{bc-ad}, \frac{bcd-abc}{bc-ad}\big)>(1,1)\cdot \big(\frac{abc-acd}{bc-ad}, \frac{bcd-abc}{bc-ad}\big)=R$. Lastly, in Region 3, $\ell_{A'}((\alpha,\beta))=(\alpha,\beta)\cdot(a,0)$ and $\alpha>\beta$, so $\ell_{A'}((\alpha,\beta))=\alpha a>2a\ge R$. Thus $\beta,\,\delta,\,\zeta\le 1$.
		
		To show that at most one of $ \beta,\,\delta,\,\gamma $ is positive, assume without loss of generality that $\beta,\,\delta \geq 1$. Another calculation as above shows that both $\ell_{A'}((\alpha,\beta))$ and $\ell_{A'}((\gamma,\delta))$ are $\ge\min\{a,d\}$, so $\ell_{A'}(\Gamma)\ge 2\min\{a,d\}\ge R$, a contradiction.
		
		A symmetric argument but with Regions 2, 3, 4, and 5 shows that $\alpha,\,\gamma,\,\epsilon\le 1$ and that at most one is positive. These facts imply that the maximum displacement in either coordinate is 1, that is, $\Gamma$ lies in $[0,1]^2$ up to translation. We check that the shortest lattice path containing 3 lattice points in $[0,1]^2$ has $\ell_{A'}$-length $R$, so $\Gamma$ cannot exist.
\end{proof}
	
	\begin{figure}[!h]
		\centering
		\begin{tikzpicture}[>=latex]
		\foreach \x in {0,1,2}
		\foreach \y in {0,1,2}
		\node at (\x,\y)[circle,fill=black][scale=0.4]{};
		\draw[->] (0,0)--(.6,.6);
		\draw (.5,.5)--(1,1);
		\draw[->](1,1)--(1,0.4);
		\draw (1,0.5)--(1,0);
		\draw[->](1,0)--(.4,0);
		\draw (.5,0)--(0,0);
		\end{tikzpicture}
		\caption{The minimal path for $c_2(A)$ in Lemma~\ref{Lem:echlemma}.}
		\label{Fig:minimalpathpic}
	\end{figure}
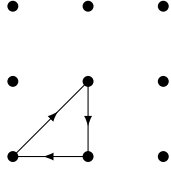

	\subsection{Toric Domains with Free Action} \label{FreeDomains}

	The proof of Corollary~\ref{Cor:NonSharp} simply combines the embeddings involved in the proof of Theorem~\ref{Thm:disk-degenerate} with the result of Hind and Lisi (\cite[Thm. 1.1]{hind}) that a symplectic embedding $P(1,2)\hookrightarrow B^4(a)$ is possible if and only if $a\geq 3.$
	
	\begin{proof}[Proof of Corollary~\ref{Cor:NonSharp}]
	
	Suppose to the contrary that $a<3$ is given for which we can find an embedding $f:P^*(1,2)\hookrightarrow B^*(a).$ Let $\lambda>1$ be close to 1 such that $\lambda^2 a<3.$ Let $P_\lambda^*(1,2)=\mu^{-1}(\mu(P(\lambda,2\lambda))+(1,1))$ and $B_\lambda^*(a)=\mu^{-1}(\mu(B^4(\lambda a))+(1,1)).$ After scaling by $\lambda,$ we can find an embedding $f_\lambda : P_\lambda^*(1,2)\hookrightarrow B_\lambda^*(a).$ This is combined with the embeddings from the proof of Theorem~\ref{Thm:disk-degenerate} as follows:
	
	First, we can find a symplectic embedding $F:P(1,2)\hookrightarrow P_\lambda^*(1,2)$ by the same technique illustrated in that theorem since $P_\lambda^*(1,2)$ is just the translated expansion of $P(1,2).$ We also have the inclusion embedding $\iota: B_\lambda^*(a)\hookrightarrow B(\lambda^2 a)$ because of the translation law (Lemma~\ref{Lem:translates}) above. Combining these we get \[\iota\circ f_\lambda\circ F: P(1,2)\hookrightarrow B(\lambda^2 a).\]
	Since $\lambda^2 a<3,$ this violates Theorem~1.1 of \cite{hind}. Thus no such embedding $f:P^*(1,2)\hookrightarrow B^*(a)$ exists.
	\end{proof}
	By Theorem~\ref{Thm:disk-degenerate}, the ECH and Ekeland-Hofer capacities of $P^*(1,2)$ and $B^*(a)$ are the same as those of $P(1,2)$ and $B(a),$ so neither of these capacities give sharp obstructions to embedding $P^*(1,2)$ into $B^*(a).$

\paragraph{Acknowledgments.}
We thank our advisor Daniel Cristofaro-Gardiner and the UC Berkeley Geometry, Topology and Operator Algebras RTG Summer Research Program for Undergraduates 2013, supported by NSF grant DMS-0838703. We also thank Michael Hutchings for helpful advice and direction for this work, and the anonymous referee for valuable feedback.

\end{document}